\documentclass{article}
\usepackage{graphicx} 

\usepackage[utf8]{inputenc}
\usepackage[english]{babel}
\usepackage{amsmath}

\usepackage{amsthm}
\usepackage{amssymb}
\usepackage{color}
\usepackage{subcaption}
\usepackage{caption}
\newtheorem{theorem}{Theorem}[section]
\newtheorem{lemma}[theorem]{Lemma}
\newtheorem{corollary}[theorem]{Corollary}
\theoremstyle{definition}

\theoremstyle{remark}

\usepackage{geometry}
\renewcommand\qedsymbol{$\blacksquare$}

\begin{document}

\title{Optimal Geodesic Curvature Constrained Dubins' Paths on a Sphere}



\author{Swaroop Darbha$^1$, Athindra Pavan$^1$, K. R. Rajagopal$^1$, Sivakumar Rathinam$^1$,\\
David W. Casbeer$^2$, 
Satyanarayana G. Manyam$^3$ \\
}

\footnotetext[1]{Texas A \& M University, College Station, TX, 77843, USA}
\footnotetext[2]{Air Force Research Laboratories, WPAFB, OH, 45431 USA}
\footnotetext[3]{Infoscitex Corp., Dayton, OH, 45431, USA}

\date{}
\maketitle

\begin{abstract}
    In this article, we consider the motion planning of a rigid object on the unit sphere with a unit speed. The motion of the object is constrained by the maximum absolute value, $U_{max}$ of geodesic curvature of its path; this constrains the object to change the heading at the fastest rate only when traveling on a tight smaller circular arc of radius $r <1$, where $r$ depends on the bound, $U_{max}$. We show in this article that if $0<r \le \frac{1}{2}$, the shortest path between any two configurations of the rigid body on the sphere consists of a concatenation of at most three circular arcs. Specifically, if $C$ is the smaller circular arc and $G$ is the great circular arc, then the optimal path can only be $CCC, CGC, CC, CG, GC, C$ or $G$. If $r> \frac{1}{2}$, while paths of the above type may cease to exist depending on the boundary conditions and the value of $r$, optimal paths may be concatenations of more than three circular arcs.
\end{abstract}



\section{Introduction}
This paper is motivated by two applications - one involving a rigid object of negligible inertia moving at a  uniform speed on a spherical surface; the second is that of a under-actuated rigid body that is spinning about one of its principal axes at a constant speed, while it can be maneuvered with limited effort about another principal axis. In both applications, the primary objective is to change the configuration of the rigid body from an initial configuration to a desired configuration in the shortest possible time/distance. The planar counterpart to the first application is that of a Dubins vehicle\cite{Dub}, where the motion is constrained by a bound on its yaw rate, and consequently, is limited by a minimum turning radius when changing direction. 

Finding the appropriate counterpart of Dubins' planar model to the 3-d case has been a challenge. Sussman \cite{SussmanHJ95} considered 3-d Dubin's problem with a curvature constraint and showed that the optimal trajectory is either a helicoidal arc or a concatenation of at most three segments (like an optimal 2-d Dubins path). For surfaces of non-negative curvature, \cite{Chitour2005DubinsPO} provides conditions on the surface for a connecting path between two configurations to exist while satisfying curvature constraints. It, however, does not deal with optimality of the path in terms of length.  In \cite{Bakolas}, a numerical technique is presented to generate optimal trajectories on a sphere in the presence of wind, however, this does not consider the geodesic curvature constraints. When the motion is constrained to the surface of a sphere, a natural generalization can be provided through the geodesic curvature constraint on the motion, where the geodesic curvature constraint is the counterpart of the planar curvature constraint on the motion of Dubins vehicle. Essentially, if the geodesic curvature is zero, the traversed path is a great circular arc ($G$) on the sphere - a geodesic. In the planar counterpart, the minimum turning radius constraint can be thought of as the corresponding geodesic curvature constraint; clearly, if the curvature is zero, the Dubins vehicle will traverse a geodesic (straight line segment). If one were to approximate a sphere as a polyhedron with sufficiently large number of facets, the Dubins path on the polyhedron can only be a small circular arc ($C$) or a straight line segment. Using the tangent plane approximation, it is not inconceivable that the optimal/shortest motion constrained path will then be a concatenation of small circular arcs and great circular arcs. Indeed, this result was shown using Pontryagin's minimum principle in \cite{MP}; more specifically, the author of \cite{MP} shows that the planar Dubins result generalizes to the spherical counterpart {\it for the specific value of smaller circular arc of radius $\frac{1}{\sqrt{2}}$}. 

The main contribution of the current article is to show that the planar Dubins result generalizes to the sphere only if the smaller circular arc can be of any radius in the region: $(0, \frac{1}{2}] \cup \{\frac{1}{\sqrt{2}}\}$; there are numerical counterexamples of $CCC$ paths not being optimal if the radius is greater than $\frac{1}{2}$ and not equal to $\frac{1}{\sqrt{2}}$.  Proof of the main result follows the treatment in \cite{Boissonat}.

The paper is organized as follows: In section 2, we provide a model for studying the motion of a Dubins vehicle on a spherical surface. In section 3, we provide the main result of the paper, and in section 4, we provide illustrative computational results that show how this problem differs from its counterpart on a plane. An appendix contains all the calculation details that are not covered in section 3. 
\section{Mathematical Formulation}
\subsection{Model}
Consider the trajectory of a rigid massless robot (modeled as a short pointed stick) on a unit sphere as shown by the green spherical curve in Fig. \ref{fig:coordsys}. 

Let $s$ denote arc length (traversed by the robot), and $\mathbf{X}(s)$ be a $C^1$-smooth {\it spherical} curve, representing the position of the robot (geometric center of the stick) on the sphere. Let ${\mathbf T}(s) := \frac{d {\mathbf X}}{ds}$ represent the direction of motion of the robot traveling with unit speed.\footnote{There is no loss in generality in considering a point traveling with unit speed on a unit sphere; one can scale distance and time appropriately to arrive at this conclusion. 
} Then, ${\mathbf T}(s)$ is a unit vector as $ds^2 = <d{\mathbf X}, d{\mathbf X}>$. Define the cross product ${\mathbf N}(s) = {\mathbf X}(s) \wedge {\mathbf T}(s)$, and the geodesic curvature, $u_g(s) := <\frac{d {\mathbf T}}{ds}, {\mathbf N}(s)>$. 
It is easy to derive that \cite{Taskopru15_SabbanOnS2}
\begin{eqnarray*}
\frac{d {\mathbf X}(s)}{ds} = {\mathbf T}(s), \quad
\frac{d {\mathbf T}(s)}{ds} = -{\mathbf X}(s) + u_g(s) {\mathbf N}(s), \quad
\frac{d {\mathbf N}(s)}{ds} = -u_g(s){\mathbf T}(s).
\end{eqnarray*}

\begin{figure}
\centering
\begin{minipage}{.45\textwidth}
  \centering
  \includegraphics[width=.9\linewidth]{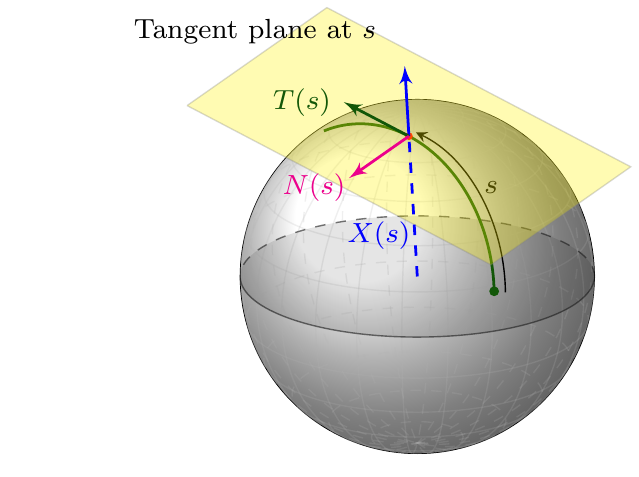}
  \captionof{figure}{Dubins model for motion on a sphere}
  \label{fig:coordsys}
\end{minipage}%
\hspace{.25cm}
\begin{minipage}{.45\textwidth}
  \centering
  \includegraphics[width=.8\linewidth]{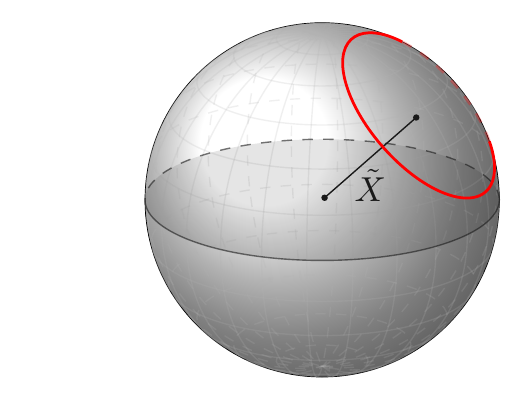}
  \captionof{figure}{Case: Constant $u_g(s) \equiv u_g \ne 0$}
  \label{fig:test2}
\end{minipage}
\end{figure}

The second derivative of $T(s)$ is 
\begin{equation} 
    \frac{d^2{\mathbf T}(s)}{ds^2} = -(1+u_g^2(s)){\mathbf T}(s)
    \label{eq:periodicT}
\end{equation}
When $u_g(s) = U$ is constant, the solution of the second order ODE in equation \eqref{eq:periodicT} is periodic with angular frequency $\omega = \sqrt{1 + U^2}$. If $U \neq 0$, the period is $\frac{2 \pi}{\sqrt{1+U^2}}$, assuming unit speed, this corresponds to a small circular arc of radius $r = \frac{1}{\sqrt{1+U^2}}$.  Furthermore, define $\tilde{\mathbf X}(s) := U {\mathbf X}(s) + {\mathbf N}(s)$. $\tilde X(s)$ lies on the radial vector that passes through the center of the small circle, as shown in the Fig. \ref{fig:test2}. Note that $\frac{d\tilde X(s)}{ds} = U \frac{dX(s)}{ds} + \frac{dN(s)}{ds} = 0$, and thus $\tilde X(s)$ is constant. If $U = 0$, the solution corresponds to a great circular arc.

\subsection{Shortest path problem formulation}


It is convenient to specify a rotation matrix ${\mathcal R}(s) \in SO(3)$ to represent the configuration of the robot; the columns of ${\mathcal R}(s)$ are respectively ${\mathbf X}(s), {\mathbf T}(s)$ and $ {\mathbf N}(s)$.

Suppose the initial configuration of the robot is $I_3$, the $3 \times 3$ identity matrix; suppose the desired final configuration is given by $R_f \in SO(3)$. The problem of determining the path of shortest length connecting the two configurations can be expressed as the following variational problem:
\begin{eqnarray}
J &=& \min \int_{0}^L 1 \;ds
\label{eq:optproblem}
\end{eqnarray}
subject to 
\begin{eqnarray}
\frac{d{\mathbf X}}{ds} = {\mathbf T}(s), \quad \quad \frac{d {\mathbf T}}{ds} = - {\mathbf X}(s) + u_g(s) {\mathbf N}(s), \quad \quad \frac{d {\mathbf N}}{ds} = - u_g(s){\mathbf T}(s), \end{eqnarray}
and the boundary conditions
\begin{eqnarray}
{\mathcal R}(0) = I_3, \quad \quad {\mathcal R}(L) = R_f.
\label{eq:optconstraints}
\end{eqnarray} 
The term $u_g(s)$ (or simply, $u_g$) is the scalar control input and represents the geodesic curvature at $s$; it is a measure of how the path is differing from the geodesic (great circle here). We assume that $u_g(s)$ is bounded by $U_{max}$, i.e., $|u_g(s)| \le U_{max}$.

The connection to underactuated spinning rigid bodies is very clear. The governing equations correspond to kinematic equations for a rigid body spinning about its $z$ axis by $1 \; rad/s$ and about its $x$ axis by $u_g(s)\; rad/sec$, which is the control effort; clearly, $|u_g(s)|\le U_{max}$ implies limited control effort can be applied to change the configuration of the body, and the overall problem is to effect a change in configuration in this case in the shortest time. 
\section{Main Results}
Let $r = \frac{1}{\sqrt{1+U_{max}^2}}$; let $C$ denote a circular arc of radius $r$ and $G$ denote a greater circular arc. We denote concatenation of arcs in the order specified by the sequence of letters; for example, a path of type $CGC$ involves concatenation of three arcs - the first arc is a smaller circular arc of radius $r$, the second one is a great circular arc and the third one is a smaller circular arc. Please note that the end points of the circular arc segment must be compatible at the points of concatenation; \textit{i.e.}, the configuration of the body at the end of the previous circular arc must be identical to the configuration at the beginning of the following circular arc. For example, in a $CGC$ path,  the configuration of the body at the end of the first smaller circular arc must be identical to the configuration at the beginning of the greater circular arc; similarly, the configuration of the body at the end of the greater circular arc must be identical to the configuration at the beginning of the second circular arc. 

The main result of this paper is as follows:
\begin{theorem}
If $0 < r \le \frac{1}{2}$, the optimal path can only be of one of the six types: CGC, CCC, CG, GC, CC, G, or C.
\label{maintheorem}
\end{theorem}
This is a generalization of characterization of optimal Dubins paths for planar systems to Dubins paths on a sphere. The main tool used is Pontryagin's minimum principle, and the proof of Theorem 3.1 uses the following summarized intermediate results:
\begin{itemize}
    \item Using Pontryagin's minimum principle, we show that the control actions are piece-wise constant and the constant can be one of the three values, i.e.,  $u_g(s) \in \{U_{max}, 0, -U_{max}\}$. This is given in Lemma \ref{lem:optcontrol}.
    \item Each of the piece-wise constant control actions results in a smaller circular arc ($C$) of radius $r$ if $u_g(s) \ne 0$ or a great circular arc of radius $1$ when $u_g(s) = 0$. This is shown in Lemma \ref{lem:arcpath}. 
    Together Lemmas \ref{lem:optcontrol} and \ref{lem:arcpath} imply that the optimal path is a concatenation of smaller circular arcs and great circular arcs.
    \item Non-optimality of a non-trivial concatenation of four circular arcs rules out the possibility of four or more circular arcs in the optimal path.  The following logic yields this result.
    \begin{itemize}
        \item If we remove the possibility of two great circular arcs in succession, since that is equivalent to one circular arc, then the eight possibilities remain: (1) $CCCC$, (2) $CCCG$, (3) $CCGC$, (4) $CGCC$, (5) $GCCC$, (6) $CGCG$, (7) $GCCG$, (8) $GCGC$.
        \item  Showing the non-optimality of paths of type $GCG$ and $GCC$ (and $CCG$ by symmetry) rules out the optimality of seven path possibilities except the $CCCC$ path. Hence, in Lemma \ref{lem:GCG_GCC}, we show the non-optimality of $GCG$ and $GCC$ paths and use Lemmas \ref{lem:inflexion}--\ref{lem:skew} for preparing the necessary background material for this lemma.
        \item Finally, Theorem \ref{th:CCCC} shows that a $CCCC$ path is not optimal if $0< r \le \frac{1}{2}$; the restriction on $r$ appears only in this part of the proof. Monroy-Perez \cite{MP} showed that the same result holds if $r = \frac{1}{\sqrt{2}}$; consequently if $r \notin \{1,\frac{1}{\sqrt{2}}\}, r > \frac{1}{2}$, it is possible that $CCCC$ path is optimal.
    \end{itemize}
\end{itemize}

\subsection{Application of Pontryagin's minimum principle}
To apply Pontryagin's minimum principle, define the Hamiltonian through the dual/adjoint variables $\lambda_1(s),$ $\lambda_2(s),$ $\lambda_3(s)$ as:
\begin{eqnarray*}
H(s; \lambda_1, \lambda_2, \lambda_3) &=& 1 + <\lambda_1, {\mathbf T}> +
<\lambda_2, -{\mathbf X} + u_g {\mathbf N}> + <\lambda_3, - u_g {\mathbf T}>\\ 
&=&1+<\lambda_1, {\mathbf T}> -<\lambda_2, {\mathbf X}> + u\{<\lambda_2, {\mathbf N}> - <\lambda_3, {\mathbf T}>\}.
\end{eqnarray*}
Define:
\begin{eqnarray}
A &:=& <\lambda_2, {\mathbf N}> - <\lambda_3, {\mathbf T}>, \\
B&:=& <\lambda_3, {\mathbf X}> - <\lambda_1, {\mathbf N}>, \\
C &:=& <\lambda_1, {\mathbf T}> - <\lambda_2, {\mathbf X}>.
\end{eqnarray}
Accordingly, 
\begin{eqnarray}
\label{eqn:adjoint}
\frac{dA}{ds} = B, \quad \quad \frac{dB}{ds} = - A + u_g C, \quad \quad \frac{dC}{ds} = - u_gB.  
\end{eqnarray}
and 
\begin{eqnarray}
H = 1 + C + u_g A.
\end{eqnarray}

The following lemma establishes the bang-bang nature of the optimal control action:
\begin{lemma}
    The optimal control action, $u_g(s)$, that solves \eqref{eq:optproblem}-\eqref{eq:optconstraints} has the following properties:
    \begin{enumerate}
        \item[(i)] 
        \begin{eqnarray}
            u_g(s) = \begin{cases}
            -U_{max}, \quad \quad A(s)>0, \\
            U_{max}, \quad  \quad A(s) <0,
            \end{cases}
        \end{eqnarray}
        \item[(ii)] $$A(s) \equiv 0, \quad \forall s \in [a,b) \implies u_g(s) \equiv 0, \quad \forall s \in [a,b). $$
    \end{enumerate}
    \label{lem:optcontrol}
\end{lemma}

\begin{proof}
\begin{enumerate}
\item[(i)]

From Pontryagin's minimum principle, $u$ minimizes $H$ pointwise; hence,
\begin{eqnarray}
u_g(s) = \begin{cases}
-U_{max}, \quad \quad A(s)>0, \\
U_{max}, \quad  \quad A(s) <0,
\end{cases}
\end{eqnarray}
and is undetermined if $A=0$. 
\item[(ii)] 
\begin{eqnarray*}
A(s) \equiv 0 \implies \frac{dA}{ds} \equiv 0 \implies B(s) \equiv 0 \implies \frac{dC}{ds} \equiv 0.
\end{eqnarray*}
This implies that $C(s) = C_0$, but by Pontryagin's Minimum Principle, $H \equiv 0 \implies C_0 = -1.$
Since $B(s) \equiv 0$, we have 
\begin{eqnarray*}
\frac{dB}{ds} = 0 \implies - A(s)+ u_g(s) C(s) = 0 \implies u_g(s) C_0 \equiv 0 \implies u_g(s) \equiv 0.
\end{eqnarray*}
\end{enumerate}
\end{proof}
Note that $u_g(s) \equiv 0$ if and only if $A(s) \equiv 0$; otherwise, $u_g(s)$ will take a value of either $-U_{max}$ or $U_{max}$. We have seen earlier in section 2.1 that $u_g(s) \equiv 0$ implies the corresponding part of the path is a great circular arc and if $u_g(s) = \pm U_{max}$, the corresponding part of the path is a small circular arc, i.e., arc of radius $r$.

\begin{lemma}
\begin{itemize}
    \item[(i)] If for all $s \in [a,b)$, $u_g(s) \equiv 0$, then the corresponding part of the path is an arc of the great circle.
    \item[(ii)] If for all $s \in [a,b)$, $|u_g(s)| = U_{max}$, then the corresponding part of the path is a small circular arc of radius $r = \frac{1}{\sqrt{1+U_{max}^2}}$. 
\end{itemize}
\label{lem:arcpath}
\end{lemma}
\begin{proof} 
See the Appendix \ref{sec:arcpath_proof}. 
\renewcommand{\qedsymbol}{}
\end{proof}
\noindent Note that $u_g(s)$ is the geodesic curvature of the path at $s$; hence, $u(s) \equiv 0$ corresponds to a great circular arc.

\begin{corollary}
From Lemmas \ref{lem:optcontrol} and \ref{lem:arcpath}, it follows that the optimal path can only consist of great circular arcs or small circular arcs.
\label{corollary}
\end{corollary}

\subsection{Non-optimality of $GCG$ \& $GCC$ paths}
From Corollary \ref{corollary}, the optimal trajectory is a concatenation of great and small circular arcs.

In this subsection, we will focus on showing non-optimality of $GCC$ and $GCG$ paths; we will show the non-optimality of $CCCC$ path in the next subsection and complete the proof of Theorem \ref{maintheorem}. To begin with, we will need to characterize the length of the arc between inflection points of the optimal path, and the following preliminaries are required for that task. 

Let $\psi(s)$ denote the vector with components $A(s), B(s)$ and $C(s)$, then the adjoint equation \eqref{eqn:adjoint} can be compactly expressed as
$$ \underbrace{\begin{pmatrix} \frac{dA}{ds} \\ \frac{dB}{ds} \\ \frac{dC}{ds} \end{pmatrix}}_{\frac{d\psi}{ds}} = \underbrace{\begin{pmatrix} 0&1&0\\-1&0&u(s)\\ 0&-u(s)& 0
\end{pmatrix}}_{\Omega(s)}\underbrace{\begin{pmatrix}A(s) \\ B(s) \\ C(s) \end{pmatrix}}_{\psi(s)}. $$
If $|u_g(s)| = U_{max}$, we can express $\Omega(s) = \Omega_i = $ on $[s_i, s_{i+1})$, then
\begin{equation}
    \psi(s) = e^{\Omega_i(s-s_i)}\psi(s_i), \quad \forall s \in [s_i, s_{i+1}).
    \label{eq:psi_sol}
\end{equation}
Corresponding to $u_g = -U_{max}$ and $U_{max}$, define respectively the following matrices
$$\Omega_L = \begin{pmatrix} 0&1&0\\-1&0&-U_{max}\\ 0&U_{max}& 0
\end{pmatrix}, \quad  \Omega_R = \begin{pmatrix} 0&1&0\\-1&0&U_{max}\\ 0&-U_{max}& 0
\end{pmatrix}.
$$
Instead of dealing with arc length, it is simpler to use the arc angle; for this reason, define 
\begin{equation}
\phi(s) := \frac{s-s_i}{r}, \quad k_x = \sqrt{1-r^2}, \quad k_y = 0, \quad k_z = r,
\label{eq:def1}
\end{equation}
and 
\begin{eqnarray}
    \hat \Omega_L &=& r \Omega_L = \begin{pmatrix} 0&k_z&0\\-k_z&0&-k_x\\ 0&k_x& 0
    \end{pmatrix}= \begin{pmatrix}
    0&r&0\\-r&0& -\sqrt{1-r^2} \\ 0 & \sqrt{1-r^2} & 0
    \end{pmatrix}, \label{eq:OhatL}\\ 
    \hat \Omega_R &=& r \Omega_R = \begin{pmatrix} 0&k_z&0\\-k_z&0&k_x\\ 0&-k_x& 0
    \end{pmatrix} =\begin{pmatrix}
    0&r&0\\ -r&0& \sqrt{1-r^2} \\ 0 & -\sqrt{1-r^2} & 0
    \end{pmatrix}. \label{eq:OhatR}
\end{eqnarray}
Hence, for $i \in \{R, L\}$, we can express $e^{\Omega_i(s-s_i)} = e^{{\hat \Omega}_i \phi}$ and employ Euler-Rodriguez formula:
\begin{equation}
    e^{{\hat \Omega}_i \phi} = I + {\hat \Omega}_i \sin \phi + {\hat \Omega}_i^2 (1- \cos \phi). 
    \label{eq:ER}
\end{equation}

The Darboux/axial vector of $\Omega_L$ and $\Omega_R$ are given by
$$ {\mathbf u}_L := \begin{pmatrix} -k_x \\ k_y \\ k_z \end{pmatrix} = \begin{pmatrix} -\sqrt{1-r^2}\\ 0 \\ r \end{pmatrix}, \quad \quad {\mathbf u}_R := \begin{pmatrix} k_x \\ k_y \\ k_z \end{pmatrix} = \begin{pmatrix} \sqrt{1-r^2}\\ 0 \\ r \end{pmatrix}. $$
Note that ${\hat \Omega}_L {\mathbf u}_L = 0$ and ${\hat \Omega}_R {\mathbf u}_R = 0$.

Let us turn our attention to inflection points on the spherical path. Inflection occurs when the control input (geodesic curvature) switches from one value to another. If inflection in the optimal path occurs at $s=s_0$, then $u(s)$ is piecewise constant on $(s_0-\epsilon, s_0)$ and $(s_0, s_0+\epsilon)$ for sufficiently small $\epsilon>0$.

\begin{lemma}
 If inflection occurs at $s=s_i$ on the optimal path, then $C(s_i) = -1$ and $A(s_i) = 0$. Furthermore, if $s_1, s_2$ are consecutive inflexion points corresponding to a CCC path, then $s_2-s_1 > \pi r$, i.e., the middle segment must have a length greater than $\pi r$. 
 \label{lem:inflexion}
\end{lemma}
 \begin{proof}
Since $A(s), B(s), C(s)$ are solutions of the ordinary differential equation \eqref{eqn:adjoint}, they are continuous in $s$. From Pontryagin's minimum principle, 
$$H(s) = 1 + C(s) + u_g(s) A(s) \equiv 0. $$
Consider an inflection point, $s_i$. Then, for $\epsilon >0$, let 
$u_g = u^{-}$ on $(s_i-\epsilon, s_i)$ and $u_g =u^{+}$ on $(s_i, s_i+\epsilon)$. Then, by Pontryagin's minimum principle,
\begin{eqnarray}
H(s_i-\epsilon) = 1 + C(s_i-\epsilon) + u^{-} A(s_i-\epsilon) = 0, \\
H(s_i + \epsilon) = 1+ C(s_i + \epsilon) + u^{+} A(s_i + \epsilon) = 0.
\end{eqnarray}
From continuity, $\lim_{\epsilon \rightarrow 0} C(s_i - \epsilon) = \lim_{\epsilon \rightarrow 0}C(s_i + \epsilon), \quad \lim_{\epsilon \rightarrow 0}A(s_i - \epsilon) = \lim_{\epsilon \rightarrow 0}A(s_i+ \epsilon)$. Since $u^{-} \ne u^{+}$, we can conclude that $C(s_i) = -1$, and $A(s_i) = 0.$

For a $CCC$ path, $|u_g(s)| = U_{max}$ for $s \in [s_1, s_2)$; hence, let $u_g = U \in \{-U_{max}, U_{max}\}$ and ${\hat \Omega}$ be defined as in \eqref{eq:OhatL} or \eqref{eq:OhatR}, depending on the value of $U$. From \eqref{eq:psi_sol}, \eqref{eq:def1} and \eqref{eq:ER} we have
\begin{eqnarray*}
\psi(s) &=& e^{{\hat \Omega} \phi} \psi(s_i) = \left(I + {\hat \Omega}\sin \phi + {\hat \Omega}^2 (1-\cos \phi)\right)\psi(s_i), \\
\implies \begin{pmatrix} A(s) \\ B(s) \\ C(s) \end{pmatrix}&=&\begin{pmatrix}
1-r^2(1-\cos \phi)&r \sin \phi & r^2U(1-\cos\phi) \\ -r \sin \phi & \cos \phi & rU \sin \phi \\ r^2U(1-\cos \phi) &-rU \sin \phi & 1-r^2U^2(1-\cos \phi)
\end{pmatrix} \begin{pmatrix} A(s_1)\\ B(s_1) \\ C(s_1) \end{pmatrix}.
\end{eqnarray*}
Noting that $A(s_1) = 0, \; C(s_1) = -1$, we have:
\begin{eqnarray*} 
B(s) 
&=& B(s_1) \cos \phi(s) - rU \sin\phi(s), \\
A(s) &=&  r\left(B(s_1) \sin \phi(s) - rU(1-\cos \phi(s)) \right).
\end{eqnarray*}
From Pontryagin's minimum principle, $UA(s) \le 0$ for $s \in [s_1,s_2)$; hence,
\begin{eqnarray}
\label{eq:Pontr}
Ur[B(s_1) \sin\phi(s)-rU(1-\cos \phi(s))]  \le 0. 
\end{eqnarray}

To prove the second part, we note that $A(s_1) = 0, A(s_2) =0$; by Rolle's Theorem, there is a ${\bar s} \in (s_1, s_2)$ such that $B({\bar s}) = A'({\bar s}) = 0$.  But,
$$B({\bar s}) = 0 \implies B(s_1) \cos \phi({\bar s}) -rU \sin\phi ({\bar s}) = 0.$$
We first note that $\cos\phi(\bar s) \ne 0$; otherwise, since $\sin\phi(\bar s) \ne 0$ simultaneously, it follows that $U = 0$, which is not the case. Hence, solving for $B(s_1)$ using the previous equation and substituting it in \eqref{eq:Pontr}, we obtain:
\begin{eqnarray}
        (Ur)^2 \left[ \tan\phi(\bar s)\sin \phi(\bar s) - (1-\cos\phi(\bar s))\right] \le 0,  
        \end{eqnarray}
or equivalently,
\begin{eqnarray}
\sin\phi(\bar s) \tan \phi(\bar s) -(1-\cos \phi(\bar s)) \le 0 
\implies \cos {\phi}({\bar s}) \le 0. 
\end{eqnarray}
This implies that $\cos \phi(\bar s)<0$ as $\cos\phi(\bar s) \ne 0$. Noting that $A(s)$ is a sinusoid, and $A'(s) = B(s)$, we infer that $B(s) = 0$ corresponds to extrema of $A(s)$; clearly if $\phi(s)$ is restricted to an interval of length $\pi$ (as $\cos \phi({\bar s}) <0$), the solution will be unique. We will show that $\bar s= \frac{1}{2}(s_1+s_2)$ and then arrive at the desired result.

Since $A(s_1) = A(s_2)$, $rU \ne 0$ and $s_2 - s_1 >0$, subtracting $A(s_1)$ from $A(s_2)$ gives
\begin{eqnarray*}
B(s_1)(\sin \phi(s_1) - \sin \phi(s_2)) + rU (\cos \phi(s_1) - \cos \phi(s_2)) = 0, \\
\implies B(s_1) \cos \left(\frac{\phi(s_1) + \phi(s_2)}{2}\right) - rU \sin \left(\frac{\phi(s_1) + \phi(s_2)}{2}\right) = 0. 
\end{eqnarray*}
Since $$\frac{\phi(s_1) + \phi(s_2)}{2} = \phi \left(\frac{s_1+s_2}{2}\right),$$ it follows that 
$$B(s_1) \cos \left(\phi\left(\frac{ s_1 + s_2}{2}\right)\right) - rU \sin \left(\phi\left(\frac{s_1 + s_2}{2}\right)\right) = 0.$$

Note that $0<s_2 -s_1 < 2 \pi r$ as the path would be non-optimal otherwise. {Since $\phi(\frac{s_1+s_2}{2}) =  \frac{s_2-s_1}{2r} \le \pi$, it follows that ${\bar s} = \frac{s_1+s_2}{2}$ as it is the only solution for $B(s) = 0$ in the interval $[s_1, s_2)$.}  {Since $\cos\phi(\bar s) <0$, it implies that $\phi(\bar s)  = \frac{s_2-s_1}{2r} \in (\frac{\pi}{2}, \frac{3 \pi}{2})$;} given that $s_2 -s_1 <2 \pi r$, this implies that 
$ \frac{s_2-s_1}{2r} \in (\frac{\pi}{2},{\pi})$; in other words, $2 \pi r > s_2-s_1 > \pi r$.

\end{proof}
\begin{lemma}
Suppose a great circular arc is part of an optimal path. Then $\|\psi(s)\|^2 = A^2(s)+B^2(s)+C^2(s) = 1$ throughout the path. 
\end{lemma}
\begin{proof}
Let the optimal path have a length $l_T$ and inflection points at $s_1, s_2, \ldots, s_L$. Define $s_0 = 0$ and $s_{L+1} =l_T$. Since the geodesic curvature is piecewise constant, let 
$$ u_g(s) = U_i, \quad \forall s \in [s_i, s_{i+1}), \; \; i=0, 1, \ldots, L.
$$
For each $i=0, \ldots, L$, $\psi(s) = e^{\Omega_i(s-s_i)} \psi(s_i)$ and by virtue of $e^{\Omega_i(s-s_i)}$ being unitary, 
$$\|\psi(s)\| = \|\psi(s_i)\|, \quad \forall s \in [s_i, s_{i+1}).
$$
Continuity of $\psi(s)$ implies continuity of $\|\psi(s)\|$, and hence, 
$\|\psi(s_{i+1})\| = \|\psi(s_i)\|$ implying that $\|\psi(s)\|$ is a constant throughout the path. Therefore, it suffices to show that $\|\psi(s_i)\| =1$ for some inflection point $s_i$. 
 
If a great circular arc is part of an optimal path, then $u_g(s)\equiv 0$ on $[s_i, s_{i+1})$ for some $i$. We also know from Lemma \ref{lem:inflexion} that $C(s_i) = -1$ and $A(s_i) =0$. 
Since $u_g(s) \equiv 0$ on $[s_i, s_{i+1})$ (otherwise, the path cannot be a greater circular arc), we must also have $A(s) \equiv 0$ (otherwise, by Pontryagin's minimum principle, we must have $u_g(s) = \pm U_{max}$, which would not correspond to a great circle). From Equation \eqref{eqn:adjoint}, it follows that $B(s) \equiv 0$ for $s \in [s_i, s_{i+1})$; in particular, $B(s_i) = 0$ and hence, $A^2(s_i) +B^2(s_i) +C^2(s_i) = 1$, implying that $A^2(s) +B^2(s)+C^2(s) \equiv 1$ throughout the path.
\end{proof}

\begin{lemma}
Suppose ${\hat \Omega}$ is a skew-symmetric matrix with ${\bf u}$ being the unit axial vector of ${\hat \Omega}$. Suppose $R = e^{{\hat \Omega} \phi}$ be a proper rotation. If ${\bf w} \ne 0$ is such that $R {\bf w} = {\bf w}$, then either $\phi=0$ or ${\bf u} = {\bf w}$ (or equivalently, ${\hat \Omega} {\bf w} = {\bf 0}$).
\label{lem:skew}
\end{lemma}
\begin{proof}
If ${\bf w} = {\bf u}$, then clearly, ${\hat \Omega} {\bf w} = {\bf 0}$; otherwise, by Rodriguez formula
$$ R{\bf w} = {\bf w} \implies [I +  {\hat \Omega} \sin \phi + {\hat \Omega}^2 (1-\cos \phi)] {\bf w} = {\bf w}. $$
 Using the vector calculus notation, this would amount to 
$$({\bf u} \times {\bf w}) \sin \phi + (1-\cos \phi) {\bf u} \times ({\bf u} \times {\bf w}) = 0. $$
Noting that the vectors ${\bf u} \times {\bf w}$ and ${\bf u} \times ({\bf u} \times {\bf w})$ are perpendicular to each other and their linear combination is zero implies that $\sin \phi = 0$ and $1- \cos\phi = 0$; hence, $\phi = 0$. 
\end{proof}

We will now show that $GCG$ and $GCC$ (and by symmetry, $CCG$) paths are not optimal, which is crucial for later proofs.

\begin{lemma}
Any non-trivial $GCG$ and $GCC$ paths are not optimal.
\label{lem:GCG_GCC}
\end{lemma}
\begin{proof}
A non-trivial path consists of arcs of non-zero length. Note that a $GCG$ path contains a great circular arc, and consequently, $A^2(s) +B^2(s) +C^2(s) \equiv 1$ throughout the path. Let the two inflexion points be $s_1$ and $s_2$ and the smaller circular arc corresponds to $s \in [s_1, s_2)$. Since $s_1, s_2$ are inflexion points, we know that $A(s_1) = 0, C(s_1) = -1$; similarly, $A(s_2) = 0, C(s_2) = -1$. This also implies that $B(s_1) = 0$ and $B(s_2) = 0$ as $A^2(s_1) + B^2(s_1) + C^2(s_1) = A^2(s_2) + B^2(s_2) +C^2(s_2) = 1$. Since $u(s) = U_2 \ne 0$ on $[s_1, s_2)$ as it corresponds to a smaller circular arc, we know that 
\begin{align}
&\begin{pmatrix} A(s_2) \\ B(s_2) \\ C(s_2) \end{pmatrix} = e^{\Omega_1(s_2-s_1)}\begin{pmatrix} A(s_1) \\ B(s_1) \\ C(s_1) \end{pmatrix} =e^{{\hat \Omega}_1\phi}\begin{pmatrix} A(s_1) \\ B(s_1) \\ C(s_1) \end{pmatrix} , \\
&\iff \underbrace{\begin{pmatrix} 0 \\ 0 \\ -1 \end{pmatrix}}_{\bf w} = e^{{\hat \Omega}_1\phi} \underbrace{\begin{pmatrix} 0 \\ 0 \\ -1 \end{pmatrix}}_{\bf w}. 
\end{align}
By the previous lemma,  ${\bf w}$ is an axial vector of $\hat \Omega_1$ or $\phi=0$. But neither of them is true, as the former would mean $r = 1$ and the latter would imply triviality of the $GCG$ path.  Hence, any non-trivial $GCG$ path cannot be optimal.

A similar reasoning holds for the $GCC$ (and by symmetry the $CCG$) path.
\end{proof}

\subsection{Non-optimality of $CCCC$ paths}
In this section, we will focus on proving that $CCCC$ paths are not optimal for $0<r\le \frac{1}{2}$. This result is formally stated in Theorem \ref{th:CCCC} at the end of the section.
The proof of this theorem is accomplished in the following steps:
\begin{itemize}
    \item First we show that the second and third circular arc lengths are equal and exceed the semi-perimeter (i.e., are of length greater than $\pi r$). This is shown in Lemma \ref{subpath}.
    \item We then show that a $CCC$ subpath of a $CCCC$ path with the last circular arcs of the same length and exceeding semi-perimeter in length cannot be optimal if $0< r \le \frac{1}{2}$. This is shown in Lemma \ref{lem:nonoptimalityofCCC}.
\end{itemize}
\begin{lemma}
\label{subpath}
For a non-trivial $CCCC$ path to be optimal, the length of the middle two arcs must be equal and greater than $\pi r$.
\end{lemma}
\begin{proof}
Let the length of the $CCCC$ path be $L$, with each segment corresponding to one of the four intervals $[s_0, s_1), [s_1, s_2), [s_2, s_3), [s_3, s_4)$, with $s_0 = 0, \quad s_4 = L$. The control input on the interval  $[s_i, s_{i+1})$ is $u_i$, with $u_i = (-1)^{i-1} U$. 

Define 
$${\hat \Omega}_1 = \begin{pmatrix}
0&k_z&0\\-k_z&0&k_x \\ 0&-k_x&0
\end{pmatrix}, \quad \quad {\hat \Omega}_2 = \begin{pmatrix}
0&k_z&0\\-k_z&0&-k_x \\ 0&k_x&0
\end{pmatrix}. $$
Note that $s_1, s_2, s_3$ are inflexion points. 
Define 
$$\phi_1 := \frac{s_1}{r} , \quad \phi_2 := \frac{(s_2-s_1)}{r}, \quad \phi_3 := \frac{(s_3-s_2)}{r}, \quad \phi_4 := \frac{(s_4-s_3)}{r}.$$
Note that $\phi_1, \phi_2, \phi_3, \phi_4$ are the angles subtended at their respective centers by the four circular arcs in the $CCCC$ path; since the circular arcs are of the same radius, it suffices to show that $\phi_2 = \phi_3$. 

Since great circular arc is not part of the CCCC path, let $B(s_1) = B_0 \ne 0$; correspondingly
$$A^2(s)+B^2(s)+C^2(s) = 1+ B_0^2. $$
At an inflexion point, $s_i$, we have $A(s_i) = 0$ and $C(s_i) = -1$. Hence, at any inflexion point, we must have $B(s_i) = \pm B_0$. Hence, at $s_2$ and $s_3$, we must have:
\begin{equation*}\begin{pmatrix}A(s_2) \\ B(s_2) \\ C(s_2)
\end{pmatrix}, \begin{pmatrix}A(s_3) \\ B(s_3) \\ C(s_3)
\end{pmatrix} \in \left\{ \begin{pmatrix} 0 \\ B_0 \\ -1
\end{pmatrix}, \begin{pmatrix} 0 \\ -B_0 \\ -1
\end{pmatrix}\right\}. 
\end{equation*}
Consider the inflexion point $s_2$; suppose
    \begin{eqnarray}
    \begin{pmatrix}A(s_2) \\ B(s_2) \\ C(s_2)
    \end{pmatrix} = \underbrace{\begin{pmatrix} 0 \\ B_0 \\ -1
    \end{pmatrix}}_{{\bf u}_1}.
    \end{eqnarray}
    However, this would imply 
    $$ {\bf u}_1 = e^{{\hat \Omega}_1 \phi_2}{{\bf u}_1}.
    $$
    In other words, ${\bf u}_1$ must be the axial vector of ${{\hat \Omega}_1}$ or $s_2-s_1=0$; neither of them is true - the former is not true because the axial vector of $\Omega_1$ is $\begin{pmatrix}\pm \sqrt{1-r^2}\\ 0 \\ r\end{pmatrix}$; the latter is not true because the trajectory is non-trivial.
    Hence, \begin{eqnarray}
    \begin{pmatrix}A(s_2) \\ B(s_2) \\ C(s_2)
    \end{pmatrix} = \underbrace{\begin{pmatrix} 0 \\ -B_0 \\ -1
    \end{pmatrix}}_{{\bf u}_2}.
    \end{eqnarray}
    Consider the inflexion point $s_3$. By the same reasoning as that for inflexion point $s_2$, we must have 
    \begin{eqnarray}
    \begin{pmatrix}A(s_3) \\ B(s_3) \\ C(s_3)
    \end{pmatrix} = \begin{pmatrix} 0 \\ B_0 \\ -1
    \end{pmatrix}.
    \end{eqnarray}
    However, this would imply that 
    $${\bf u}_2 = e^{{\hat \Omega}_1 \phi_2}{\bf u}_1, \quad {\bf u}_1 = e^{{\hat \Omega}_2 \phi_3}{\bf u}_2. $$
    Combining, 
    $$[e^{{\hat \Omega}_1 \phi_2} - e^{-{\hat \Omega}_2 \phi_3}]{\bf u}_1 = 0. $$
    Using Euler-Rodriguez formula for exponential of a skew-symmetric matrix for $e^{{\hat \Omega}_1 \phi_2}$ and $e^{-{\hat \Omega}_2 \phi_3}$, we get:
    \begin{align*}
 &   \begin{pmatrix}
1-k_z^2(1-\cos \phi_2)&k_z \sin \phi_2 & k_xk_z(1-\cos\phi_2) \\ -k_z \sin \phi_2 & \cos \phi_2 & k_x \sin \phi_2 \\ k_xk_z(1-\cos \phi_2) &-k_x \sin \phi_2 & 1-k_x^2(1-\cos \phi_2)
\end{pmatrix}\begin{pmatrix} 0 \\ B_0 \\ -1 \end{pmatrix} =\\ 
&\begin{pmatrix}
1-k_z^2(1-\cos \phi_3)& -k_z \sin \phi_3 & -k_xk_z(1-\cos\phi_3) \\ k_z \sin \phi_3 & \cos \phi_3 & k_x \sin \phi_3 \\ -k_xk_z(1-\cos \phi_3) &-k_x \sin \phi_3 & 1-k_x^2(1-\cos \phi_3)
\end{pmatrix}\begin{pmatrix} 0 \\B_0 \\ -1 \end{pmatrix}.
    \end{align*}
    The last two equations of the above set of equations are:
    \begin{eqnarray*}
    B_0 (\cos \phi_2 - \cos \phi_3) -k_x( \sin \phi_2 - \sin \phi_3) &=& 0, \\
    -B_0k_x(\sin \phi_2 - \sin \phi_3) - k_x^2(\cos \phi_2 - \cos \phi_3) &=& 0.
    \end{eqnarray*}
    Since $k_x \ne 0$, 
    $$\begin{pmatrix}
    B_0 & -k_x \\ k_x & B_0 
    \end{pmatrix} \begin{pmatrix}
    \cos \phi_2 - \cos \phi_3 \\ \sin \phi_2 - \sin \phi_3
    \end{pmatrix} = \begin{pmatrix}
    0\\ 0
    \end{pmatrix},$$
    this implies $ \cos \phi_2 - \cos \phi_3 = 0$ and $\sin \phi_2 - \sin \phi_3 = 0$ as $B_0^2+k_x^2 \ne 0$!
    Consequently, $\phi_2 = \phi_3$.

     \end{proof}
   To prove the theorem, it is necessary to distinguish a circular arc by its orientation, i.e., whether $u(s) = U_{max}$ or $u(s) = -U_{max}$. An arc, $C$, is of type L if corresponding $u_g(s) = -U_{max}$ and of type $R$ if corresponding $u_g(s) = U_{max}$. Hence, we can be more specific and distinguish between concatenations; for example, while $LGR$, $LGL$, $RGL$, $RGR$ are all paths of type $CGC$, they are clearly four different paths. To make paths and the corresponding arc angles explicit, we write $L_{\alpha}R_{\beta}L_{\gamma}$ for a concatenated path of three arcs, the first of which is of type $L$ and has an angle $\alpha$, the second arc is of type $R$ and of angle $\beta$ and the third is of type $L$ and of angle $\gamma$.
   
   If we had a non-trivial $CCCC$ path, it can be of type $LRLR$ or $RLRL$; optimality of path necessitates that two central arcs have arc lengths exceeding $\pi r$ and be equal. To show non-optimality, it suffices to consider non-optimality of paths of type $LRLR$ (and by reflection symmetry, the result also holds for $RLRL$). To show non-optimality of $LRLR$, it suffices to show that a subpath of $LRLR$ consisting of only three segments is non-optimal. 
   Consider a subpath, $L_{\alpha}R_{\pi+\phi}L_{\pi+\phi}$, with the angle of the first arc $0<\alpha <<1$ while the second and third arcs be $\pi + \phi$, with $0<\phi< \pi$.  We will show that this path is non-optimal if $0<r \le \frac{1}{2}$ to conclude the proof of Theorem 2.
   
   \begin{lemma}
   \label{lem:nonoptimalityofCCC}
    For any $r\in(0,\frac{1}{2}]$, $\alpha>0$, and $\phi \in (0, \pi)$, the path $L_{\alpha}R_{\pi+\phi}L_{\pi+\phi}$ is not optimal; in particular, there is a path of type $RLR$ with smaller length certifying its non-optimality.
   \end{lemma}
   
    \begin{proof}

    Without any loss of generality, we may assume $0<\alpha <<1$; otherwise, there is a subpath of the path with this property which we can show is non-optimal. We may employ regular perturbation technique to study the non-optimality of $L_{\alpha}R_{\pi+\phi}L_{\pi+\phi}$ for the case $\alpha <<1$. It is easier to see that 
    angles corresponding to arcs in the RLR paths may be chosen to be $\pi + \phi + \xi(\alpha), \pi + \phi + \eta(\alpha)$ and $\beta(\alpha)$ respectively, with $\beta(0) = 0, \xi(0) = 0, \eta(0) = 0$. Note that the angles $\xi(\alpha),\eta(\alpha),\beta(\alpha)$ can be expressed as functions of the perturbation variable $\alpha$.

       Define 
       $$R_L(\alpha) = e^{{\hat \Omega}_L \alpha}, \quad \quad R_R(\alpha) = e^{{\hat \Omega}_R\alpha}. $$ 
   The following equation must hold for the final configurations of $LRL$ and $RLR$ paths to be the same starting from the same initial configuration:
    \begin{equation} \label{eqn:Main CCCC equation}
    \begin{split}
    R_L(\alpha)R_R(\pi+\phi)R_L(\pi+\phi) = R_R(\pi+\phi+\xi(\alpha))R_L(\pi+\phi+\eta(\alpha))R_R(\beta(\alpha)).
    \end{split}
    \end{equation}
    Essentially, the above equation implies that if the configuration is the same before and after the three turns, whether you take a $LRL$ path or a $RLR$ path; however, the total distance traveled by the paths can be different. To determine non-optimality of $L_{\alpha}R_{\pi+\phi}L_{\pi+\phi}$, one must establish that the difference, $\Delta(\alpha)$, between the lengths of $LRL$ and $RLR$ paths i.e., 
    \begin{align*}
    \Delta (\alpha) := &\alpha + (\pi + \phi) + (\pi + \phi) - ((\pi+\phi + \xi(\alpha)) + (\pi + \phi + \eta(\alpha) + \beta(\alpha)) \\
    = &\alpha - \xi(\alpha) - \eta(\alpha) - \beta(\alpha). 
    \end{align*}
    The $LRL$ path will not be optimal if,
    for sufficiently small $\alpha >0$, $\Delta (\alpha)>0$  or equivalently
    \begin{eqnarray}
    \label{eq:Comparison}
    0 > \xi(\alpha) + \beta(\alpha) + \eta(\alpha) - \alpha
    \end{eqnarray} when equation \eqref{eqn:Main CCCC equation} is satisfied. 
    
    
    We will employ Taylor's series expansion for $\xi(\alpha), \eta(\alpha)$ and $\beta(\alpha)$ to check if inequality \eqref{eq:Comparison} holds for sufficiently small $\alpha>0$. Noting that $\xi(0) = 0, \eta(0) = 0, \beta(0) = 0$ from equation \eqref{eqn:Main CCCC equation}, let
    \begin{align*}
    \xi(\alpha) := &a_1 \alpha + \frac{1}{2}b_1 \alpha^2 + \ldots, \\
    \eta(\alpha) := &a_2 \alpha + \frac{1}{2} b_2 \alpha^2 + \ldots \\
    \beta(\alpha) := &a_3 \alpha + \frac{1}{2} b_3 \alpha^2 + \ldots.
    \end{align*}
    It suffices to show that either (a) $a_1+a_2+a_3<1$ or (b) $a_1 + a_2 + a_3 =1$ and $b_1+b_2+b_3<0$ for equation \eqref{eq:Comparison} to hold. Essentially, the latter case indicates that the lengths of $LRL$ and $RLR$ paths are the same in the first order approximation, but the $LRL$ path is longer in the second order approximation. In what follows, we will show that the latter holds for the case at hand.

    Differentiating equation \eqref{eqn:Main CCCC equation} with respect to $\alpha$, we obtain:
    \begin{align}
    \nonumber
        {\hat \Omega}_L R_L(\alpha)R_R(\pi+\phi) & R_L(\pi+\phi) = \\
        &\frac{d \xi}{d \alpha}{\hat \Omega}_R R_R(\pi+\phi+\xi(\alpha))R_L(\pi+\phi+\eta(\alpha))R_R(\beta(\alpha)) \nonumber \\
        \nonumber
        + &\frac{d \eta}{d \alpha} R_R(\pi+\phi+\xi(\alpha))R_L(\pi+\phi+\eta(\alpha)){\hat \Omega}_LR_R(\beta(\alpha)) \\
        + &\frac{d \beta}{d \alpha} R_R(\pi+\phi+\xi(\alpha))R_L(\pi+\phi+\eta(\alpha))R_R(\beta(\alpha)){\hat \Omega}_R.
        \label{eqn:Main_First_Order}
    \end{align}
   Evaluating both sides of equation \eqref{eqn:Main_First_Order} at $\alpha = 0$:
    \begin{equation} \label{eqn:First Order}
    \begin{split}
    &{\hat \Omega}_L R_R(\pi+\phi)R_L(\pi+\phi)=a_1 {\hat \Omega}_RR_R(\pi+\phi)R_L(\pi+\phi)\\
    &+a_2R_R(\pi+\phi)R_L(\pi+\phi){\hat \Omega}_L+a_3R_R(\pi+\phi)R_L(\pi+\phi){\hat \Omega}_R.
    \end{split}
    \end{equation}

    It is helpful to note that the axial vectors ${\mathbf u}_L, {\mathbf u}_R$ satisfy the following properties that will aid in setting up the equations to solve for $a_1, a_2$ and $a_3$:
    $${\hat \Omega}_L {\bf u}_L = 0, \quad R_L(\alpha){\mathbf u}_L = {\mathbf u}_L, \quad {\hat \Omega}_R {\bf u}_R = 0, \quad R_R(\alpha){\mathbf u}_R = {\mathbf u}_R. $$
    Pre-multiplying both sides of the equation \eqref{eqn:First Order} by ${\mathbf u}_L^T$ and post-multiplying by ${\mathbf u}_R$, we obtain:
    \begin{equation}
        \label{eqn:First Order eq 1}
        0= a_1 \underbrace{{\mathbf u}_L^T{\hat \Omega}_RR_R(\pi+\phi)R_L(\pi+\phi){\mathbf u}_R}_{A_{100LR}} + a_2 \underbrace{{\mathbf u}_L^TR_R(\pi+\phi)R_L(\pi+\phi){\hat \Omega}_L {\mathbf u}_R}_{A_{010LR}}.
    \end{equation}
    In the Appendix \ref{sec:appndxlem10}, we show that 
    $$A_{100LR} = A_{01LR} = 4r^2(1-r^2) \sin \phi (\cos \phi + (2r^2-1)(1+ \cos \phi)).$$
    If $0 <r \le \frac{1}{2}$, and $\phi \ne 0$, $A_{010LR} \ne 0$ as:
   \begin{align*}
   \cos \phi + (2r^2-1)(1+ \cos \phi) = & -1 + 2r^2(1+ \cos \phi) \\
    \le & -\frac{1-\cos \phi}{2} <0. 
   \end{align*}
    Hence, 
    \begin{equation} a_1 + a_2 = 0, \quad {\textit or} \quad a_1 = -a_2
    \end{equation}
    We obtain equation \eqref{eqn:First Order eq 2} by pre-multiplying both sides of equation \eqref{eqn:First Order} with ${\mathbf u}_R^T$ and post-multiplying by ${\mathbf u}_L$:
    \begin{eqnarray}
    \nonumber
    {\mathbf u}_R^T{\hat \Omega}_L R_R(\pi+\phi)R_L(\pi+\phi){\mathbf u}_L &=& a_3 {\mathbf u}_R^TR_R(\pi+\phi)R_L(\pi+\phi){\hat \Omega}_R {\mathbf u}_L, \\
      \label{eqn:First Order eq 2}
     \implies \underbrace{{\mathbf u}_R^T{\hat \Omega}_L R_R(\pi+\phi){\mathbf u}_L}_{A_{000RL}}  &=&   a_3 \underbrace{{\mathbf u}_R^TR_L(\pi+\phi){\hat \Omega}_R {\mathbf u}_L}_{A_{001RL}}.
    \end{eqnarray}
    In the Appendix \ref{sec:appndxlem10}, we  show that 
    $$A_{000RL} = A_{001RL} = -4r^2(1-r^2)\sin \phi.$$
    Hence 
    $$a_3 =1 \implies a_1+a_2+a_3 = 1. $$
    In the first order approximation $\Delta(\alpha) = (a_1+a_2+a_3-1)\alpha = 0.$
    The second order approximation of $\Delta(\alpha)$ will require the values of $a_1$ and $a_2$ and hence we now set out to find $a_1$ and $a_2$ by setting up equation
     \eqref{eqn:First Order eq 3} by pre-multiplying both sides of equation \eqref{eqn:First Order} with ${\mathbf u}_R^T$ and post-multiplying by ${\mathbf u}_R$:
     \begin{align}
     \label{eqn:First Order eq 3}
     &{\mathbf u}_R^T{\hat \Omega}_L R_R(\pi+\phi)R_L(\pi+\phi) {\mathbf u}_R = a_2 {\mathbf u}_R^TR_R(\pi+\phi)R_L(\pi+\phi){\hat \Omega}_L {\mathbf u}_R, \\
     & \implies 
     \underbrace{{\mathbf u}_R^T{\hat \Omega}_L R_R(\pi+\phi)R_L(\pi+\phi) {\mathbf u}_R}_{A_{000RR}} = a_2 \underbrace{{\mathbf u}_R^TR_L(\pi+\phi){\hat \Omega}_L {\mathbf u}_R}_{A_{010RR}}.
     \end{align}
     In the Appendix \ref{sec:appndxlem10}, we show that 
     $$A_{000RR} = 4r^2(1-r^2) \sin \phi (1-2r^2(1+\cos \phi)), \, A_{001RR} = 4r^2(1-r^2) \sin \phi,
     $$
     implying that 
     \begin{eqnarray}
     a_2 = (1-2r^2(1+\cos \phi)) = -a_1. 
     \end{eqnarray}

    We also notice that $-a_1 = a_2 >0$ for $r \in [0, \frac{1}{2}]$ and $\phi \in (0, \pi)$.
    
    The following equation, obtained by differentiating both sides of equation \eqref{eqn:First Order} with respect to $\alpha$ and evaluating at $\alpha = 0$, is useful for obtaining the set of linear equations to determine  $b_1, b_2, b_3$: 
    \begin{equation} \label{eqn:Second Order}
    \begin{split}
    &{\hat \Omega}_L^2R_R(\pi+\phi)R_L(\pi+\phi) \\ 
    &=b_1[{\hat \Omega}_RR_R(\pi+\phi)R_L(\pi+\phi)]+a_1^2[{\hat \Omega}_R^2R_R(\pi+\phi)R_L(\pi+\phi)]\\
    &+a_1a_2[{\hat \Omega}_RR_R(\pi+\phi)R_L(\pi+\phi)\Omega_L]+a_1a_3[{\hat \Omega}_RR_R(\pi+\phi)R_L(\pi+\phi)\Omega_R]\\
    &+b_2[R_R(\pi+\phi)R_L(\pi+\phi){\hat \Omega}_L]+a_1a_2[{\hat \Omega}_RR_R(\pi+\phi)R_L(\pi+\phi){\hat \Omega}_L]\\
    &+a_2^2[R_R(\pi+\phi)R_L(\pi+\phi){\hat \Omega}_L^2]+a_2a_3[R_R(\pi+\phi)R_L(\pi+\phi){\hat \Omega}_L{\hat \Omega}_R]\\
    &+b_3[R_R(\pi+\phi)R_L(\pi+\phi){\hat \Omega}_R]+a_1a_3[{\hat \Omega}_RR_R(\pi+\phi)R_L(\pi+\phi){\hat \Omega}_R]\\
    &+a_2a_3[R_R(\pi+\phi)R_L(\pi+\phi){\hat \Omega}_L{\hat \Omega}_R]+a_3^2[R_R(\pi+\phi)R_L(\pi+\phi){\hat \Omega}_R^2].
    \end{split}
    \end{equation}
    As in the first order case, we pre-multiply both sides of the above equation \eqref{eqn:Second Order} by  ${\mathbf u}_L^T$  and post-multiply by ${\mathbf u}_R$ to get the first equation, and pre-multiply by ${\mathbf u}_R^T$ and post-multiply by ${\mathbf u}_L$ to get the second equation. These two equations will suffice to get $(b_1+b_2+b_3)$ as the first equation will provide, as in the first order case, the value of $b_1+b_2$, while the second one will provide $b_3$. 
    Similar to the first order case, the following set of equations provide the necessary set of linear equations to determined $b_1,b_2,b_3$:
    \begin{align} 
    \nonumber
    &0 = b_1 \underbrace{{\bf u}_L^T {\hat \Omega}_R R_R(\pi+ \phi) R_L(\pi+ \phi){\bf u}_R}_{A_{100LR}}
    + b_2 \underbrace{{\bf u}_L^T R_R(\pi+ \phi) R_L(\pi+ \phi){\hat \Omega}_L{\bf u}_R}_{A_{010LR}} \\
    \nonumber
    &+ a_1^2 \underbrace{{\bf u}_L^T {\hat \Omega}_R^2 R_R(\pi+ \phi) R_L(\pi+ \phi){\bf u}_R}_{A_{200LR}}
    + 2a_1a_2 \underbrace{ {\bf u}_L^T {\hat \Omega}_R R_R(\pi+ \phi) R_L(\pi+ \phi) {\hat \Omega}_L{\bf u}_R}_{A_{110LR}}\\
    &+ a_2^2 \underbrace{{\bf u}_L^T R_R(\pi+ \phi) R_L(\pi+ \phi) {\hat \Omega}_L^2{\bf u}_R}_{A_{020LR}},
    \label{eqn:Second order eq1}
    \end{align}
    \begin{align}
    \nonumber
    &\underbrace{{\bf u}_R^T {\hat \Omega}_L^2R_R(\pi+\phi) R_L(\pi+\phi){\bf u}_L}_{B_{000RL}} = b_3\underbrace{{\bf u}_R^T R_L(\pi+\phi){\hat \Omega}_R{\bf u}_L}_{A_{001RL}} \\ 
    &+ 2a_2a_3 \underbrace{{\bf u}_R^T R_R(\pi+\phi)R_L(\pi+\phi){\hat \Omega}_L{\hat \Omega}_R{\bf u}_L}_{A_{011RL}} 
     +a_3^2 \underbrace{{\bf u}_R^T R_L(\pi+\phi){\hat \Omega}_R^2{\bf u}_L}_{A_{002RL}},
    \label{eqn:Second order eq2}
    \end{align}
   We have already seen that 
    \begin{eqnarray}
    \label{eqn:A100LR}
    A_{100LR} = A_{010LR} = 4r^2(1-r^2)a_1\sin \phi.
    \end{eqnarray}
    In the Appendix \ref{sec:appndxlem10}, we show that 
    \begin{eqnarray}
    \label{eqn:A200LR}
    A_{200LR} &=& -4r^2(1-r^2)[1-c^2\phi + (1-2r^2)\cos \phi(1+\cos \phi))],\\
    \label{eqn:A110LR}
    A_{110LR} &=& 4r^2(1-r^2)[\cos^2\phi+(1-2r^2)(1-\cos^2 \phi)], \\
    \label{eqn:A020LR}
    A_{020LR} &=& -4r^2(1-r^2)[1-\cos^2 \phi + (1-2r^2)\cos \phi (1+\cos \phi)],\\
    \label{eqn:B000RL}
    B_{000RL} &=& 4r^2(1-r^2)(1-2r^2)(1+\cos \phi), \\
    \label{eqn:A001RL}
    A_{001RL} &=& -4r^2(1-r^2) \sin \phi, \\
    \label{eqn:A011RL}
    A_{011RL} &=& -4r^2(1-r^2)\cos \phi. \\
    \label{eqn:A002RL}
    A_{002RL} &=& 4r^2(1-r^2)(1-2r^2)(1+\cos \phi).
    \end{eqnarray}
    Noting that $a_1=-a_2$, equation \eqref{eqn:Second order eq1} together with the equations \eqref{eqn:A200LR} to \eqref{eqn:A011RL} yields:
\begin{eqnarray*}
(b_1+b_2) \sin \phi &=& -a_1(A_{200LR} -2 A_{110LR} + A_{020LR}), \\
&=& 2a_1(1+(1-2r^2)(1+ \cos \phi)).
\end{eqnarray*}
A similar simplification of equation \eqref{eqn:Second order eq2} with the knowledge that $a_3 = 1$ yields:
$$b_3 \sin \phi = 2a_1 \cos \phi. $$

Hence, 
$$(b_1+b_2+b_3) \sin \phi = 4a_1 (1+\cos \phi)(1-r^2). $$
Since $a_1 <0$ whenever $r \in [0, \frac{1}{2}]$ and $\phi \in (0, \pi)$, it follows that $b_1+b_2+b_3<0$; hence,  for sufficiently small $\alpha>0$,  ${\cal L}_{\alpha}{\cal R}_{\pi+\phi}{\cal L}_{\pi+\phi}$ path is not optimal, thereby completing the proof.

\end{proof}

\begin{theorem} If $0 < r \le \frac{1}{2}$, then any non-trivial  CCCC path cannot be optimal. 
\label{th:CCCC}
\end{theorem}
\begin{proof}
The proof follows directly from Lemmas \ref{subpath} and \ref{lem:nonoptimalityofCCC}.
\end{proof}

\section{Results and Conclusions}
In the previous section, we have provided a proof for the optimality of $CGC$ and $CCC$ paths or their subpaths when $r \le \frac{1}{2}$. A natural question arises as to what happens when $r > \frac{1}{2}$. Intuitively, as $r \rightarrow 1$, the maneuverability of the Dubins' vehicle on the unit sphere decreases. Correspondingly, a change in the configuration of the Dubins' vehicle will necessitate travelling a longer distance for attaining the desired change in configuration. In the limit, the distance will actually reach $\infty$; e.g., consider the case of a Dubins' vehicle starting at the North Pole pointing eastward. Suppose the vehicle is desired to be at the same location pointing westward. Clearly, when $r=1$ (same as the radius of the great circle), there is no possibility of accomplishing this change in configuration; no path can exist. Correspondingly, the distance becomes $\infty$. If $r<1$, it is natural to ask whether a three segment path exists? To answer this question, see the illustration shown in Fig. \ref{fig:existence}.
\begin{figure}[htpb]
\includegraphics[width=\linewidth]{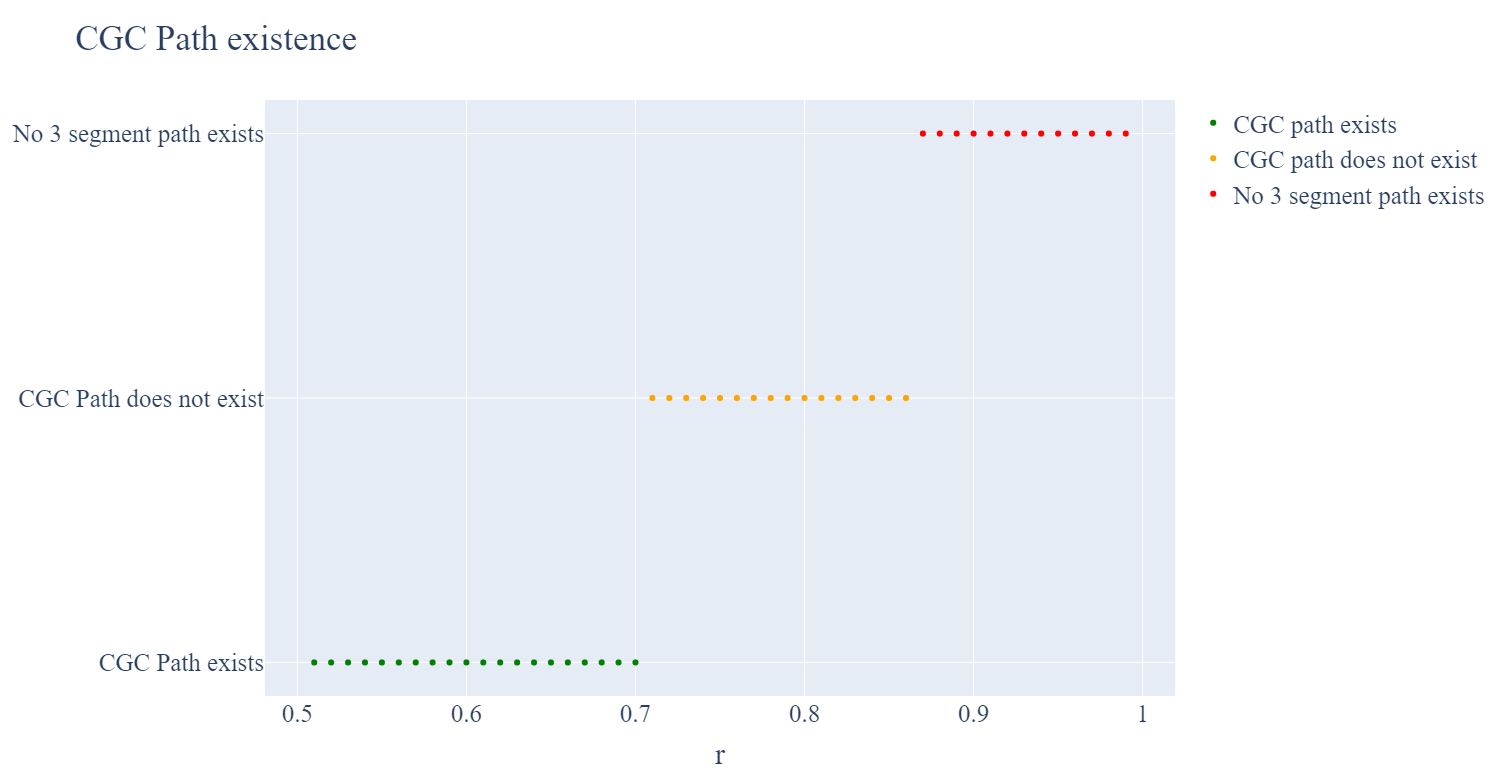}
\caption{Existence of three-segment paths connecting two configurations with the same location (North Pole), but having opposite heading}
\label{fig:existence}
\end{figure}
Computations seem to suggest that a $CGC$ type path exists only till $r \le \frac{1}{\sqrt{2}} \approx 0.707$;  furthermore no three-segment path exists beyond $r = \frac{\sqrt{3}}{2}$. A typical $CCC$ path corresponding to $r = \frac{\sqrt{3}}{2}$ is shown in the Fig. \ref{fig:teardrop}. 
\begin{figure}[htpb]
\centering
\includegraphics[scale=0.75]{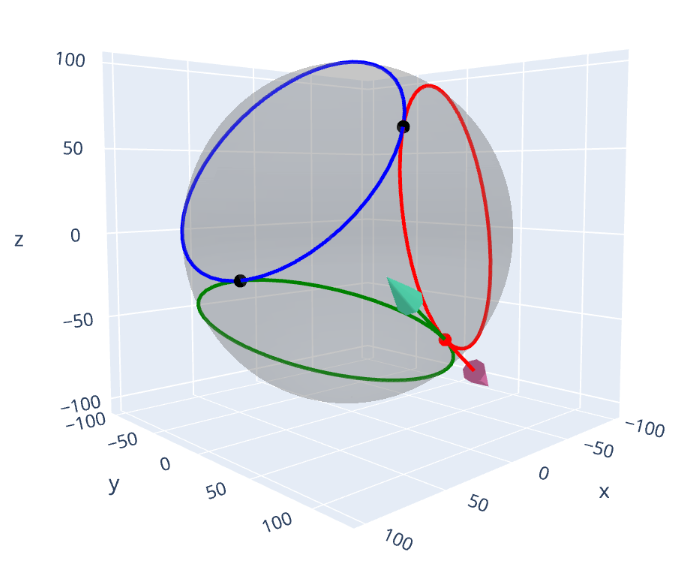}
\caption{A CCC path for Dubins vehicle for $r= \frac{\sqrt{3}}{2}$}
\label{fig:teardrop}
\end{figure}
Corresponding to the maximum possible value of $r$ at which a $CCC$ path exists, each of the circular arcs will correspond to exactly $\pi$ radians and the corresponding points of tangency will lie on a great circle, the diameters of the small circular arcs will correspond to the sides of the equilateral triangle with points of tangency as their vertices. For this reason, beyond $r= \frac{\sqrt{3}}{2}$, no $CCC$ path can exist for this pair of boundary conditions. Clearly, limitations of maneuverability manifest in terms of non-existence of a $CCC$ or a $CGC$ path. For some $p \ge 4$, there may be a $p$-segment path that connects these configurations; in this sense, the Dubins' paths on a sphere differ from those in a plane. 

There are other interesting boundary conditions that relate to how the optimal solution bifurcates - this is at the heart of the proof of Lemma \ref{lem:nonoptimalityofCCC}. In this proof, an optimal solution of the type $R_{\pi+\phi}L_{\pi+\phi}$ was perturbed to a path $L_{\alpha}R_{\pi+\phi}L_{\pi+\phi}$; in this case, we showed that there was a solution of type $R_{\phi+\psi(\alpha)} L_{\phi+\eta(\alpha)}R_{\beta(\alpha)}$ that had a smaller length when $r \le \frac{1}{2}$. In the Fig. \ref{fig:char}, $\alpha = 1^{\circ}$, $\phi \in (0, \pi)$ and $r \in (0,1)$; specifically, we considered discrete values of $\phi$ ranging from $2^{\circ}$ to $178^{\circ}$ in steps of $2^{\circ}$, and determine the optimal paths connecting the boundary configurations, $R(0) = I_3$ and $R(L) = L_{\alpha}R_{\pi+\phi}L_{\pi+\phi}$. Similarly, we considered discrete values of $r$ ranging from $0.01$ to $0.99$ in steps of $0.01$. In this plot, each dot corresponds to a specific value of $r$ and $\phi$. Each dot corresponds to an instance specified by $r$ and $\phi$. The color coding of the dot depends on whether the optimal path for the boundary conditions considered is of type $LRL$ or $RLR$ or $CGC$. The legend specifies the association of the color with the path. What can be inferred is that when $r \le \frac{1}{2}$ either $RLR$ or $CGC$ path is optimal, thereby corroborating the non-optimality of $L_{\alpha}R_{\pi+\phi}L_{\pi+\phi}$ path. 
\begin{figure}[htpb]
\centering{\includegraphics[width=\linewidth]{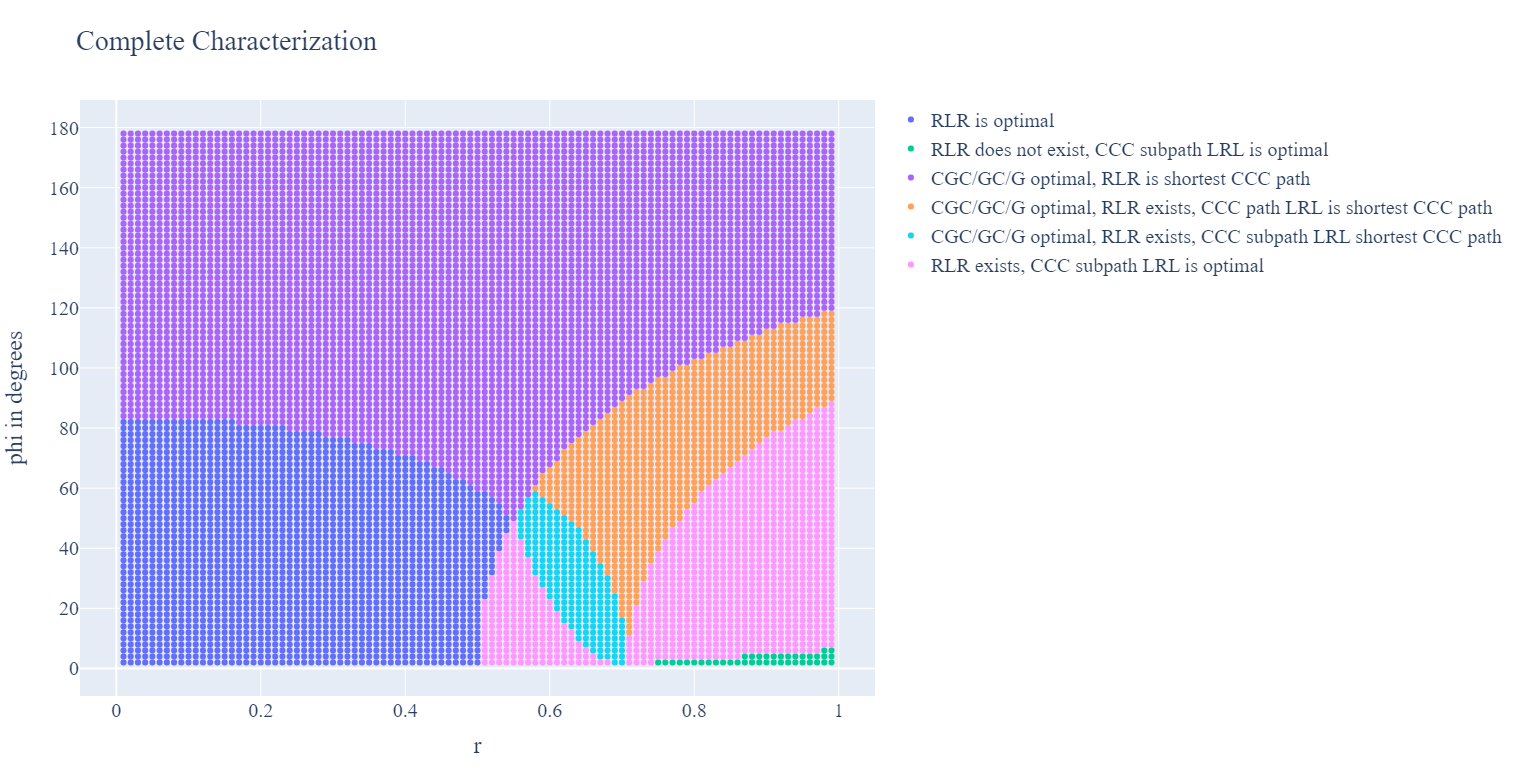}
}
\caption{Numerical characterization of optimal path for a set of boundary conditions given by $R(0)= I_3, \quad R(L) = L_{\alpha} R_{\pi+\phi}L_{\pi+\phi}$.}
\label{fig:char}
\end{figure}

\bibliographystyle{plain}
\bibliography{references}

\section{Appendix}

\subsection{Proof of Lemma \ref{lem:arcpath}}
\label{sec:arcpath_proof}
\begin{proof}[Proof of Lemma \ref{lem:arcpath}]
\begin{itemize}
\item[(i)] This can be seen from the governing equations:
$$\frac{dX}{ds} = T(s), \quad \quad \frac{dT}{ds} = -X(s), \quad \quad \frac{dN}{ds} = 0. $$
Hence, $N(s) = N_0$ is the normal to the plane containing the great circle, and 
\begin{eqnarray}
X(s) = \begin{pmatrix} 
\cos(s) \\ - \sin(s) \\ 0
\end{pmatrix}, \quad \quad 
T(s) = \begin{pmatrix} 
-\sin(s) \\ - \cos(s) \\ 0
\end{pmatrix}.
\end{eqnarray}
\item[(ii)] Define ${\tilde X}(s) := u(s) X(s) + N(s)$. For $s \in [a,b)$, we note that  $\frac{d{\tilde X}(s)}{ds} = U \frac{dX}{ds} + \frac{dN}{ds} = 0$. Hence, ${\tilde X}(s)$ remains constant on $[a,b)$. 

Note that $<{\tilde X}(s), T(s)> = 0$; define ${\tilde N}(s) = {\tilde X}(s) \times T(s) = UN(s) - X(s)$. Then:
$$\frac{d{\tilde X}}{ds} = 0, \quad \quad \frac{dT(s)}{ds} = {\tilde N}(s), \quad \quad \frac{d{\tilde N}(s)}{ds} = -(1+U^2)T(s). $$
Clearly, then
\begin{eqnarray}
{\tilde X}(s) &=& {\tilde X}(a), \\
T(s) &=& T(a) \cos (s \sqrt{1+U^2}) + {\tilde N}(a) \sin (s \sqrt{1+U^2})\\
{\tilde N}(s) &=& \sqrt{1+U^2}[-T(a) \sin (s\sqrt{1+U^2}) + {\tilde N}(a) \cos(s \sqrt{1+U^2})].
\end{eqnarray}
Clearly, the motion is periodic and the length of the period (circumference of the smaller circle) is $\frac{2\pi}{\sqrt{1+U^2}}$. The motion of the object corresponding to $s \in [a,b)$ is a circular arc of radius $r=\frac{1}{\sqrt{1+U^2}}$ and is in the plane with a normal ${\tilde X}(a)$.
\end{itemize}
\end{proof}

\subsection{Identities needed for proof of Lemma \ref{lem:nonoptimalityofCCC}} \label{sec:appndxlem10}
We will encounter the following quantities frequently and hence, we list them for easy reference:
\begin{eqnarray}
{\hat \Omega}_L = \begin{pmatrix} 0&k_z& 0 \\ -k_z & 0 & k_x \\ 0& -k_x & 0 \end{pmatrix} =\begin{pmatrix} 0& -r& 0\\
                           r& 0& -\sqrt{1-r^2}\\
                           0& \sqrt{1-r^2}& 0\end{pmatrix} ;
    \\
   {\hat  \Omega}_R = \begin{pmatrix} 0&k_z& 0 \\ -k_z & 0 & -k_x \\ 0& k_x & 0 \end{pmatrix} = \begin{pmatrix} 0& -r& 0\\
                           r& 0& \sqrt{1-r^2}\\
                           0& -\sqrt{1-r^2}& 0
    \end{pmatrix}.
\end{eqnarray}

As we have seen before, the axial vectors of ${\hat \Omega}_L$ and ${\hat \Omega}_R$ are respectively given by:
\begin{equation}
{\bf u}_L=\begin{pmatrix} \sqrt{1-r^2}\\0\\r\end{pmatrix}; \quad
{\bf u}_R=\begin{pmatrix} -\sqrt{1-r^2}\\0\\r\end{pmatrix}.
\end{equation}
Furthermore, the corresponding rotation matrices are:
\begin{eqnarray*}
R_L(\pi+\phi) &=& I-{\hat\Omega}_L\sin\phi+ {\hat \Omega}_L^2(1+\cos\phi),\\
R_R(\pi+\phi) &=& I- {\hat\Omega}_R\sin\phi+{\hat \Omega}_R^2(1+\cos\phi).
\end{eqnarray*}
Let
\begin{equation}
e_2 = \begin{pmatrix}0\\ 1\\ 0\end{pmatrix}\quad\quad
\implies e_2^T = \begin{pmatrix}0& 1& 0\end{pmatrix}.
\end{equation}
The following relationships will be useful in solving for $a_1, a_2, a_3, b_1, b_2$ and $b_3$:
\begin{equation}
{\hat \Omega}_L{\bf u}_R=-2r\sqrt{1-r^2}e_2,
\quad\quad
{\hat \Omega}_R{\bf u}_L=2r\sqrt{1-r^2}e_2.
\nonumber
\end{equation}

\begin{equation}
\Omega_L{\bf u}_L=\Omega_R{\bf u}_R= 0 \quad  \implies {\bf u}_L^T\Omega_L={\bf u}_R^T\Omega_R=0
\nonumber
\end{equation}
\begin{equation}
\begin{split}
{\bf u}_R^TR_R(\pi+\phi)={\bf u}_R^T,\quad\quad {\bf u}_L^TR_L(\pi+\phi)={\bf u}_L^T, \\
R_R(\pi+\phi){\bf u}_R={\bf u}_R,\quad\quad R_L(\pi+\phi){\bf u}_L={\bf u}_L.
\nonumber  
\end{split}
\end{equation}

\noindent{\bf Claim 1:}
\begin{itemize}
    \item If ${\hat \Omega} \in \{{\hat \Omega}_L, {\hat \Omega}_R\}$, then $e_2^T{\hat \Omega}^2e_2 = -1$.
    \item $e_2^T{\hat \Omega}_R{\hat \Omega}_Le_2 = 1-2r^2 = e_2^T{\hat \Omega}_L{\hat \Omega}_Re_2$.
    \item $e_2^T{\hat \Omega}_R^2{\hat \Omega}_L e_2 = 0$.
\end{itemize}
\begin{proof}
\begin{align*}
& {\hat \Omega} e_2 = \begin{pmatrix} k_z \\ 0 \\ \pm k_x \end{pmatrix} \implies e_2^T {\hat \Omega}^T{\hat \Omega}e_2 = 1 \iff e_2^T{\hat \Omega}^2 e_2 = -1.  \\
& {\hat \Omega}_L e_2 = \begin{pmatrix} k_z \\ 0 \\  k_x \end{pmatrix}, \quad {\hat \Omega}_L e_2 = \begin{pmatrix} k_z \\ 0 \\ -k_x \end{pmatrix} \\
& \implies  e_2^T{\hat \Omega}_R^T{\hat \Omega}_:e_2 = k_z^2-k_x^2 = 2r^2-1 \iff e_2^T{\hat \Omega}_R{\hat \Omega}_Le_2 = 1-2r^2.\\
& {\hat \Omega}_R{\hat \Omega}_L e_2 = \begin{pmatrix} 0&k_z& 0 \\ -k_z& 0 &k_x \\ 0&-k_x&0
\end{pmatrix}\begin{pmatrix} k_z \\ 0 \\ -k_x
\end{pmatrix} = -e_2  \\
&\iff e_2{\hat \Omega}_R^2{\hat \Omega}_Le_2 = -e_2^T{\Omega}_Re_2 = 0.
\end{align*}
It is easy to see that $e_2^T{\hat \Omega}_L{\hat \Omega}_R e_2 = 1-2r^2$ and 
$e_2^T{\hat \Omega}_L^2{\hat \Omega}_Re_2 = 0$.
\end{proof}

\noindent{\bf Claim 2:} 
\begin{enumerate}
    \item $A_{100LR} = {\mathbf u}_L^T{\hat \Omega}_R R_R(\pi + \phi) R_L(\pi + \phi) {\mathbf u}_R =4r^2(1-r^2)a_1\sin \phi$.
    \item $ A_{010LR} = {\mathbf u}_L^T R_R(\pi + \phi) R_L(\pi + \phi){\hat \Omega}_L {\mathbf u}_R=4r^2(1-r^2)a_1\sin \phi.$
\end{enumerate}
\begin{proof}
Recognizing that $${\hat \Omega}_R R_R(\pi + \phi) = \frac{dR_R(\pi+ \phi)}{d \phi},$$
we obtain using Euler-Rodriguez formula that
\begin{align*}
&A_{100LR}= {\mathbf u}_L^T[-{\hat \Omega}_R^T \cos \phi - {\hat \Omega}_R^2 \sin \phi][I-{\hat \Omega}_L \sin \phi + {\hat \Omega}_L^2 (1+\cos \phi)] {\mathbf u}_R, \\
&= 2k_xk_z [e_2^T\cos \phi + e_2^T {\hat \Omega}_R\sin \phi][{\mathbf u}_R +2k_xk_z (e_2 \sin \phi - {\hat \Omega}_Le_2 (1+ \cos \phi) ].
\end{align*}
Since $e_2$ is perpendicular to both ${\mathbf u}_R$ and ${\mathbf u}_L$, and ${\hat \Omega}_R{\mathbf u}_R = 0$, 
we get 
\begin{align*}
    A_{100LR} &= 4k_x^2k_z^2(\cos \phi \sin \phi- e_2^T{\hat \Omega}_R {\hat \Omega}_Le_2 \sin \phi (1+ \cos \phi) \\
    &= 4r^2(1-r^2)\sin \phi [ \cos \phi + (2r^2-1) (1+ \cos \phi)].
\end{align*}

Using a similar approach:
\begin{eqnarray*}
A_{010LR} &=& {\mathbf u}_L^T R_R(\pi + \phi) R_L(\pi + \phi){\hat \Omega}_L {\mathbf u}_R \\
&=& {\mathbf u}_L^T [I- {\hat \Omega}_R \sin \phi + {\hat \Omega}_R^2 (1+ \cos \phi)] [-{\hat \Omega}_L \cos \phi - {\hat \Omega}_L^2 \sin \phi] {\mathbf u}_R\\
&=& [{\mathbf  u}_L^T - {\mathbf u}_L^T {\hat \Omega}_L \sin \phi + {\mathbf u}_L^T {\hat \Omega}_L^2 (1+ \cos \phi)][-{\hat \Omega}_R {\mathbf u}_L \cos \phi -{\hat \Omega}_R^2 {\mathbf u}_L \sin \phi] \\
&=&[ - {\mathbf u}_L^T {\hat \Omega}_R \sin \phi + {\mathbf u}_L^T {\hat \Omega}_R^2 (1+ \cos \phi)][-{\hat \Omega}_L {\mathbf u}_R \cos \phi -{\hat \Omega}_L^2 {\mathbf u}_R \sin \phi] \\
&=& (2k_xk_z)[e_2^T \sin \phi -e_2^T {\hat \Omega}_R(1+\cos \phi)] (2k_xk_z) [e_2 \cos \phi + {\hat \Omega}_L e_2] \\
&=& 4k_x^2k_z^2[\cos \phi \sin \phi +(2r^2-1) \sin \phi (1+ \cos \phi)], \\
&=& 4r^2(1-r^2)\sin \phi [ \cos \phi + (2r^2-1) (1+ \cos \phi)].
\end{eqnarray*}
\end{proof}
\noindent{\bf Claim 3:} 
\begin{enumerate}
    \item $A_{000RL} := {\mathbf u}_R^T{\hat \Omega}_L R_R(\pi + \phi) {\mathbf u}_L = -4r^2(1-r^2) \sin \phi.$
    \item $A_{001RL} := {\mathbf u}_R^TR_L(\pi + \phi) {\hat \Omega}_R{\mathbf u}_L = -4r^2(1-r^2) \sin \phi.$
\end{enumerate}
\begin{proof}
    \begin{align*}
    A_{000RL}&=[{\mathbf u}_R^T{\hat \Omega}_L][R_R(\pi + \phi) {\mathbf u}_L] \\
    &=2k_xk_z e_2^T[{\mathbf u}_L - {\hat \Omega}_R{\mathbf u}_L\sin \phi + {\hat \Omega}_R^2{\mathbf u}_L (1+ \cos \phi)]\\
    &= 2k_xk_ze_2^T[-2k_xk_z (e_2 \sin \phi - {\hat \Omega}_R e_2 (1+ \cos \phi)] \\
    &= -4k_x^2k_z^2 \sin \phi = -4r^2(1-r^2) \sin \phi.
    \end{align*}
    Similarly,
    \begin{align*}
    A_{001RL} &= [{\mathbf u}_R^TR_L(\pi + \phi)][{\hat \Omega}_R{\mathbf u}_L] \\
    &= 2k_xk_z[{\mathbf u}_R^T - {\mathbf u}_R^T{\hat \Omega}_L \sin\phi + {\mathbf u}_R^T{\hat \Omega}_L^2(1+\cos \phi)]e_2\\
    &= 2k_xk_z[- {\mathbf u}_R^T{\hat \Omega}_L \sin\phi + {\mathbf u}_R^T{\hat \Omega}_L^2(1+\cos \phi)]e_2\\
    &= -4k_x^2k_z^2[e_2^T\sin \phi-e_2^T {\hat \Omega}_L(1+ \cos \phi)]e_2\\
    &= -4k_x^2k_z^2 \sin \phi = -4r^2(1-r^2) \sin \phi.
    \end{align*}
\end{proof}
\noindent{\bf Claim 4:} 
\begin{enumerate}
    \item $A_{000RR}:= {\mathbf u}_R^T{\hat \Omega}_LR_R(\pi+\phi)R_L(\pi+\phi){\mathbf u}_R = 4r^2(1-r^2)\sin \phi[1-2r^2(1+\cos\phi)].$
    \item $A_{010RR} := {\mathbf u}_R^TR_L(\pi + \phi) {\hat \Omega}_L(\pi + \phi){\mathbf u}_R = 4r^2(1-r^2) \sin \phi.$
\end{enumerate}
\begin{proof}
    \begin{align*}
    &A_{000RR}= [{\mathbf u}_R^T{\hat \Omega}_L]R_R(\pi+\phi)][]R_L(\pi+\phi){\mathbf u}_R]\\ 
    =& 2k_xk_z e_2^T[I-{\hat \Omega}_R \sin \phi + {\hat \Omega}_R^2 (1+\cos \phi)][{\mathbf u}_R -{\hat \Omega}_L {\mathbf u}_R\sin \phi + {\hat \Omega}_L^2{\mathbf u}_R (1+\cos \phi)] \\
    =& 2k_xk_z[e_2^T - e_2^T{\hat \Omega}_R\sin \phi + e_2^T{\hat \Omega}_R^2 (1+\cos \phi)][-{\hat \Omega}_L {\mathbf u}_R\sin \phi + {\hat \Omega}_L^2{\mathbf u}_R (1+\cos \phi)]\\
    =& 2k_xk_z[e_2^T - e_2^T{\hat \Omega}_R\sin \phi + e_2^T{\hat \Omega}_R^2 (1+\cos \phi)] (2k_xk_z)[e_2\sin \phi - {\hat \Omega}_Le_2 (1+\cos \phi)]\\
    =& 4k_x^2k_z^2[ \sin \phi + \underbrace{e_2^T{\hat \Omega}_R{\hat \Omega}_Le_2}_{1-2r^2} \sin \phi (1+\cos \phi)+ \underbrace{e_2^T{\hat \Omega}_R^2e_2}_{-1} \sin \phi (1+ \cos \phi) \\ 
    &- \underbrace{e_2^T{\hat \Omega}_R^2{\hat \Omega}_Le_2}_{0}(1+\cos \phi)^2]\\
    =& 4k_x^2k_z^2 \sin \phi[ 1+ (1-2r^2) (1+ \cos \phi)-(1+ \cos \phi)]\\
    =& 4r^2(1-r^2)\sin \phi[1-2r^2(1+\cos\phi)].
    \end{align*}
    Similarly,
    \begin{align*}
    A_{010RR} &= {\mathbf u}_R^T[[R_L(\pi + \phi) {\hat \Omega}_L(\pi + \phi)] {\mathbf u}_R] = 
    {\mathbf u}_R^T[-{\hat \Omega}_L {\mathbf u}_R \cos \phi - {\hat \Omega}_L^2 {\mathbf u}_R\sin \phi] \\
    &= -{\mathbf u}_R^T {\hat \Omega}_L^2 {\mathbf u}_R \sin \phi = 4r^2(1-r^2) \sin \phi
    \end{align*}
\end{proof}
\noindent{\bf Claim 5:}
\begin{enumerate}
    \item $A_{200LR} := {\mathbf u}_L^T{\hat \Omega}_R^2 R_R(\pi+ \phi) R_L(\pi+\phi){\mathbf u}_R = -4r^2(1-r^2)[1-\cos^2 \phi + (1-2r^2) \cos \phi (1+ \cos \phi)].$
     \item $A_{110LR} := {\mathbf u}_L^T{\hat \Omega}_R R_R(\pi+ \phi) R_L(\pi+\phi){\hat \Omega}_L {\mathbf u}_R = 4r^2(1-r^2)[\cos^2 \phi + (1-2r^2) (1- \cos^2 \phi)].$
     \item $A_{020LR} := {\mathbf u}_L^T R_R(\pi+ \phi) R_L(\pi+\phi){\hat \Omega}_L^2 {\mathbf u}_R = -4r^2(1-r^2)[1-\cos^2 \phi + (1-2r^2) \cos \phi (1+ \cos \phi)].$
\end{enumerate}
\begin{proof}
\begin{eqnarray*}
A_{200LR} &:=& [{\mathbf u}_L^T[{\hat \Omega}_R^2 R_R(\pi+ \phi)]][ R_L(\pi+\phi){\mathbf u}_R ] \\
&=& [{\mathbf u}_L^T {\hat \Omega}_R \sin \phi - {\mathbf u}_L^T {\hat \Omega}_R^2 \cos \phi ][{\mathbf u}_R - {\hat \Omega}_L {\mathbf u}_R \sin \phi +
{\hat \Omega}_L^2 {\mathbf u}_R (1+\cos \phi)]\\
&=&2k_xk_z[-e_2^T \sin \phi +e_2^T{\hat \Omega}_R \cos \phi][{\mathbf u}_R +2k_xk_z[e_2 \sin \phi - {\hat \Omega}_Le_2(1+ \cos \phi)]]\\
&=& (2k_xk_z)^2[-e_2^T \sin \phi +e_2^T{\hat \Omega}_R \cos \phi][e_2 \sin \phi - {\hat \Omega}_Le_2(1+ \cos \phi)]\\
&=& -4k_x^2k_z^2[\sin^2 \phi + e_2^T{\hat \Omega}_R{\hat \Omega}_L e_2\cos \phi (1+ \cos \phi)] \\ 
&=& -4r^2(1-r^2)[1-\cos^2 \phi + (1-2r^2) \cos \phi (1+ \cos \phi)].
\end{eqnarray*}
Similarly,
\begin{eqnarray*}
A_{110LR} &:=& [{\mathbf u}_L^T[{\hat \Omega}_R R_R(\pi+ \phi)]] [[R_L(\pi+\phi){\hat \Omega}_L] {\mathbf u}_R] \\
&=& [-{\mathbf u}_L^T {\hat \Omega}_R \cos \phi -{\mathbf u}_L^T {\hat \Omega}_R^2 \sin \phi ][- {\hat \Omega}_L{\mathbf u}_R \cos \phi - 
{\hat \Omega}_L^2{\mathbf u}_R \sin \phi]\\
&=& (2k_xk_z)^2[e_2^T \cos \phi  + e_2^T {\hat \Omega}_R \sin \phi][e_2 \cos \phi + {\hat \Omega}_Le_2 \sin \phi] \\
&=& 4k_x^2k_z^2[ \cos^2 \phi + e_2^T{\hat \Omega}_R{\hat \Omega}_L e_2 \sin^2 \phi] \\
&=& 4r^2(1-r^2)[\cos^2 \phi + (1-2r^2) (1- \cos^2 \phi)].
\end{eqnarray*}
\begin{eqnarray*}
A_{020LR} &:=& [{\mathbf u}_L^T R_R(\pi+ \phi)] [[R_L(\pi+\phi){\hat \Omega}_L^2] {\mathbf u}_R] \\
&=&  [{\mathbf u}_L^T - {\mathbf u}_L^T {\hat \Omega}_R \sin \phi + {\mathbf u}_L^T {\hat \Omega}_R^2 (1+ \cos \phi)] [{\hat \Omega}_L {\mathbf u}_R\sin \phi - {\hat \Omega}_L^2 {\mathbf u}_R\cos \phi]\\
&=&  (2k_xk_z)^2[e_2^T \sin \phi -e_2^T{\hat \Omega}_R (1+ \cos \phi)] [-e_2\sin \phi + {\hat \Omega}_Le_2 \cos \phi]\\
&=& 4k_x^2k_z^2[ - \sin^2 \phi -e_2^T{\hat \Omega}_R{\hat \Omega}_L e_2 \cos \phi (1+ \cos \phi)] \\
&=& -4r^2(1-r^2)[1-\cos^2 \phi + (1-2r^2) \cos \phi (1+ \cos \phi)]
\end{eqnarray*}
\end{proof}
\noindent{\bf Claim 6:}
\begin{enumerate}
    \item $B_{000RL} := {\mathbf u}_R^T{\hat \Omega}_L^2R_R(\pi + \phi)R_L(\pi + \phi){\mathbf u}_L = 4r^2(1-r^2)(1-2r^2)(1+\cos \phi)$.
    \item $A_{011RL} := {\mathbf u}_R^TR_R(\pi + \phi)R_L(\pi+\phi){\hat \Omega}_L{\hat \Omega}_R{\mathbf u}_L = -4r^2(1-r^2) \cos \phi$.
    \item $A_{002RL} := {\mathbf u}_R^T R_L(\pi + \phi){\hat \Omega}_R^2{\mathbf u}_L = 4r^2(1-r^2)(1-2r^2)(1+ \cos \phi).$
\end{enumerate}
\begin{proof}
\begin{eqnarray*}
B_{000RL} &:=& [{\mathbf u}_R^T{\hat \Omega}_L^2]R_R(\pi + \phi)[R_L(\pi + \phi){\mathbf u}_L] \\
&=& [{\mathbf u}_R^T{\hat \Omega}_L^2][R_R(\pi + \phi){\mathbf u}_L] \\
&=& 2k_xk_z [e_2^T{\hat \Omega}_L][{\mathbf u}_L - {\hat \Omega}_R {\mathbf u}_L \sin \phi + 
{\hat \Omega}_R^2 {\mathbf u}_L (1+\cos \phi)]\\
&=& 2k_xk_z e_2^T{\hat \Omega}_L {\hat \Omega}_R^2 {\mathbf u}_L (1+ \cos \phi) \\ 
&=& 2k_xk_z e_2^T{\hat \Omega}_L (2k_x k_z {\hat \Omega}_Re_2 (1+ \cos \phi)) \\
&=& 4k_x^2k_z^2 e_2^T{\hat \Omega}_L{\hat \Omega}_Re_2 (1+\cos \phi) \\
&=& 4r^2(1-r^2)(1-2r^2)(1+\cos \phi).
\end{eqnarray*}
\begin{eqnarray*}
A_{011RL} &:=& [{\mathbf u}_R^TR_R(\pi + \phi)][R_L(\pi+\phi){\hat \Omega}_L][{\hat \Omega}_R{\mathbf u}_L] \\
&=& [{\mathbf u}_R^T[R_L(\pi+\phi){\hat \Omega}_L]][{\hat \Omega}_R{\mathbf u}_L] \\
&=& [ -{\mathbf u}_R^T {\hat \Omega}_L \cos \phi  -{\mathbf u}_R^T {\hat \Omega}_L^2 \sin \phi] (2k_xk_z e_2)\\
&=&2k_xk_z[-2k_xk_z e_2^T \cos \phi -2k_xk_ze_2^T {\hat \Omega}_L] e_2\\
&=& -4k_x^2k_z^2 \cos \phi =  -4r^2(1-r^2) \cos \phi
\end{eqnarray*}
Finally, 
\begin{eqnarray*}
A_{002RL} &:=& [{\mathbf u}_R^T R_L(\pi + \phi)][{\hat \Omega}_R^2{\mathbf u}_L] \\
&=& [{\mathbf u}_R^T - {\mathbf u}_R^T {\hat \Omega}_L \sin \phi + {\mathbf u}_R^T {\hat \Omega}_L^2 (1+\cos \phi)]{\hat \Omega}_R [2k_xk_z e_2]
\\
&=& 2k_xk_z [-2k_xk_z e_2^T \sin \phi +2k_xk_z e_2^T {\hat \Omega}_L (1+ \cos \phi)]{\hat \Omega}_Re_2 \\
&=& 4k_x^2k_z^2(1-2r^2) \sin \phi (1+ \cos \phi) \\
&=& 4r^2(1-r^2)(1-2r^2)(1+ \cos \phi).
\end{eqnarray*}
\end{proof}
\end{document}